\documentclass{article}
\usepackage{psfrag}
\usepackage[parfill]{parskip}
\usepackage{float, graphicx}
\usepackage[]{epsfig}
\usepackage{amsmath, amsthm, amssymb}
\usepackage{epsfig}
\usepackage{verbatim}
\usepackage{multicol}
\usepackage{url}
\usepackage{latexsym}
\usepackage{mathrsfs}
\usepackage[colorlinks, bookmarks=true]{hyperref}
\usepackage{graphicx}
\usepackage{enumerate}
\usepackage{bm}
\usepackage{dsfont}
\usepackage{stmaryrd}

\usepackage[utf8]{inputenc}

\usepackage[maxnames=7]{biblatex}
 
\usepackage{color}
\makeatother

\newcommand{\EE}[1]{\mathbb{E}\left[ #1 \right]}
\newcommand{\bE}{\mathbb{E}}
\newcommand{\bT}{\mathbb{T}}
\newcommand{\PP}[1]{\mathbb{P}\left[ #1 \right]}
\newcommand{\bP}{\mathbb{P}}
\newcommand{\bC}{\mathbb{C}}
\newcommand{\bZ}{\mathbb{Z}}
\newcommand{\cK}{\mathcal{K}}

\newcommand{\re}{\mathrm{e}}
\newcommand{\ri}{\mathrm{i}}
\newcommand{\bR}{{\mathbb R}}
\DeclareMathOperator{\Tr}{Tr}

\DeclareMathOperator{\supp}{supp}
\DeclareMathOperator{\dist}{dist}
\renewcommand{\Re}{\mathop{\mathrm{Re}}}
\renewcommand{\Im}{\mathop{\mathrm{Im}}}
\newcommand{\dd}{\mathrm{d}}

\newcommand{\floor}[1] {\lfloor {#1} \rfloor}

\newcommand{\norm}[1]{\lVert #1 \rVert}
\newcommand{\abs}[1]{\vert #1 \vert}

\newcommand{\zn}{\mathbb{Z}/N\mathbb{Z}}
\newcommand{\inv}[1]{\frac{1}{#1}}
\newcommand{\cnn}{\bC^{N\times N}}

\theoremstyle{plain} 
\newtheorem{theorem}{Theorem}[section]
\newtheorem*{theorem*}{Theorem}
\newtheorem{lemma}[theorem]{Lemma}
\newtheorem*{lemma*}{Lemma}
\newtheorem{corollary}[theorem]{Corollary}
\newtheorem*{corollary*}{Corollary}
\newtheorem{proposition}[theorem]{Proposition}
\newtheorem*{proposition*}{Proposition}

\newtheorem*{assumption*}{Assumption}

\newtheorem*{definition*}{Definition}

\newtheorem*{example*}{Example}
\newtheorem{remark}[theorem]{Remark}

\newtheorem*{remark*}{Remark}
\newtheorem*{remarks*}{Remarks}

\def\author#1{\par
    {\centering{\authorfont#1}\par\vspace*{0.05in}}
}

\def\titlefont{\fontsize{13}{15}\bfseries\boldmath\selectfont\centering{}}
\def\authorfont{\fontsize{13}{15}}

\let\affiliationfont\rhfont

\def\address#1{\par
    {\centering{\affiliationfont#1\par}}\par\vspace*{11pt}
}

\def\body{
\setcounter{footnote}{0}
\def\thefootnote{\alph{footnote}}
\def\@makefnmark{{$^{\rm \@thefnmark}$}}
}

\def\title#1{
    \thispagestyle{plain}
    \vspace*{-14pt}
    \vskip 79pt
    {\centering{\titlefont #1\par}}%
    \vskip 1em
}

\begin{document}

\title{The Edge Universality of Correlated Matrices}%
\vspace{1.2cm}
\noindent \begin{minipage}{0.5\textwidth}
\author{Arka Adhikari}%
\address{Harvard University\\
adhikari@math.harvard.edu}%
 \end{minipage}
\noindent \begin{minipage}{0.5\textwidth}
\author{Ziliang Che }
\address{Harvard University\\
zche@math.harvard.edu}%
 \end{minipage}
 ~\vspace{0.3cm}
\begin{abstract}
We consider a Gaussian random matrix with correlated entries that have a power law decay of order $d>2$ and prove universality for the extreme eigenvalues. A local law is proved using the self-consistent equation combined with a decomposition of the matrix.  This local law along with concentration of eigenvalues around the edge allows us to get an bound for extreme eigenvalues. Using a recent result of the Dyson-Brownian motion, we prove universality of extreme eigenvalues. 
\end{abstract}
{\let\thefootnote\relax\footnote{Z.C. is partially supported by NSF grant DMS-1607871.}}

\begingroup
\hypersetup{linkcolor=black}
 \tableofcontents
\endgroup

\date{\today}

\newpage
\section{Introduction}

The Wigner-Dyson-Mehta conjecture asserts that the local eigenvalue statistics of large random matrices are universal in the sense that they depend only on the symmetry class of the model - real symmetric or complex Hermitian - but are otherwise independent of the underlying details of the model. There are two types of universality results. Bulk universality involves the spacing distribution eigenvalues that lie well within the support of the limiting spectral distribution, while edge universality involves the extreme eigenvalues.

There has recently been a lot of progress made in proving the Wigner-Dyson-Mehta conjecture in a increasingly large class of models.  In \cite{Erdos2012a,Erdos2013a,Erdos2009b,Erdos2012c,Erdos2012b,Erdos2010universality}, universality was proved for Wigner matrices whose entries are independent and have identical variance; parallel results are obtained independently in various cases in \cite{Tao2010,Tao2011a}. In \cite{Ajanki2015,ziliang}, this type of result was extended to more general variance patterns, while still maintaining the independence of matrix entries. 

Most of the previous works rely heavily on the independence between matrix entries, and  deal with bulk universality. Only recently have people proved results on models with general correlation structure. In \cite{Che2016, Erdos2015, Ajanki2016}, bulk universality is proved for matrices where the correlation decays fast enough. In a recent paper \cite{Erdos2017d}, Erd{\"o}s et al. consider a model where the correlation between matrix entries has a power law decay of order $d\geq 12$ in the long range and $d\geq 2$ in the short range. They use a combinatorial expansion to get optimal local law, then prove bulk universality. They remark in Example 2.12 that in the Gaussian case, $d\geq 2$ for both long range and short range correlation is sufficient to satisfy the assumptions of their main theorem.

In this paper, we prove edge universality for Gaussian matrices with a correlation structure that decays as a power law of order $d>2$, namely $|\EE{h_{ij}h_{kl}}|\leq\frac{1}{|i-l|^{d} + |j-k|^{d}}$ where $h_{ij}$ are the entries of the random matrix $H$. Our proof avoids the expansion of Greens function, but relies on a decomposition of Gaussian random variables into a sum of short range interactions. 

Recent proofs of universality have followed a robust three step strategy:
\begin{enumerate}
    \item Prove a local law for the empirical eigenvalue distribution at small scales.
    \item  Study the convergence of the DBM (Dyson-Brownian motion) in short time scales to local equilibrium.
    \item  Prove that the eigenvalue spacing distribution does not change too much during the short time evolution of DBM.
\end{enumerate}
Step 1, finding the local law, is generally the most difficult and model dependent.
The strategy in proving this local law is deriving a self-consistent equation for the Green's function $G=(H-z)^{-1}$.

One can heuristically derive a self consistent equation by taking expectation and performing integration by parts on $G(H-z)=I$. One notices that there is a linear operator $S$ such that $\EE{G(-S(G)-z)} =1$. Removing the expectation creates some error term. The goal is to show that a small error exists with high probability on our matrix ensemble, as is done in \cite{Erdos2015,Che2016}.

From \cite{Che2016}, it is known that the self-consistent equation for correlated matrix entries is of the form $G(- S(G)-z ) = I$ that can be transformed into the following vector equation via local Fourier transform.
\begin{gather}
    g(x)(-\Psi(g)(x)-z)=1,\quad x\in L^\infty([0,1]^2)
\end{gather}
where $\Psi:L^\infty([0,1]^2)\to L^\infty([0,1]^2)$ is an integral operator, which is the continuous version of $S$.
There are two difficulties in our case: getting a small error for our self-consistent equation and proving the stability of the equation near the edge.

In order to get a small error for the self-consistent equation, we avoid the procedure of removing blocks of elements, which requires combinatorial expansion, but instead applied integration by parts and concentration results along a careful decomposition of the probability space. This gives us a weak local law which can be bootstrapped to give an even better bound for the expected value of the Green's function. Once we have bounds on the expected value, we use the concentration of eigenvalues about its mean value in order to show a version of upper bound for the top eigenvalue along the edge.

In order to prove the stability, we first embed the matrix space into the continuous space $C^\infty([0,1]^2)$, up to small errors. However, entry-wise error is not small enough to allow this embedding. We noticed the fact that the operator $S$ has a smoothing effect and will reduce the error; thus, a double iteration of the operator $F(G)=(-S(G)-z)^{-1}$ created a matrix $F(F(G))$ that satisfies
\begin{equation}
    F(F(G)) = F(F(F(G))) + R,
\end{equation}
where $R$ has sufficiently fast decay on off-diagonal entries. A similar strategy based on the smoothing effect of $F$ is also used in \cite{Ajanki2016}.  Then we can embed and apply stability of the continuous solution. In order to prove the decay properties of the double iteration, we applied a preturbation around a fixed matrix that is known to have decay of matrix entries. With sufficiently strong upper bounds on the top eigenvalue and lower bounds on the bottom eigenvalue, we are able to use the result of \cite{Yau2017} to get universality for the extreme eigenvalues. The result of \cite{Erdos2017d} is sufficient to locate the extremal eigenvalues but we have an approach that allows us to get optimal correlation decay without  a combinatorial expansion.

The structure of this paper is as follows. The second section is devoted to proving a self-consistent equation with sufficiently small error. The third section of this paper involves proving stability of the self-consistent equation to get a local law to prove an upper bound on eigenvalues. The final section uses this upper bound in order to prove universality.

\textbf{Acknowledgements}: We thank J. Huang and H-T Yau for useful discussions. The argument of Sec 3.3. came from a private communication with the two.


\section{Derivation of self-consistent equation}

\subsection{The Model and Assumptions}

For $N \in\mathbb{N}$, we consider a symmetric matrix $H = (h_{ij}^{(N)})_{1\leq i,j\leq N}$ whose entries are centered Gaussian random variables. For simplicity of notation we omit the dependence of $h_{ij}$ on $N$.  Let $\xi_{ijkl}:= N\EE{ h_{ij}h_{kl}}$. Assume there is a Lipschitz function $\phi: \bT \times \bZ \to \bR$ such that 
\begin{equation}\label{def:xi}
    \xi_{ijkl} = \phi(i/N, j/N, k-i,l-j)+O(N^{-1}),\forall i\leq k, j\leq l.
\end{equation}
Let $\bZ_N = \bZ/N\bZ$, and from now on we view the indices $i,j,k,l$ as elements in $\zn$. On  $\zn$ we define the natural distance $\dist_{\zn} (i,j):= \min\{\abs{i-j+kN}|k\in\bZ\}$, which for simplicity of notation we still denote by $\abs{i-j}$ unless there is danger of confusion. Assume that there are universal constants  $d > 2$ and $\mathsf{c}_1>0$ such that 
\begin{equation}\label{def:xidecay}
    \abs{\xi_{ijkl}} \leq \mathsf{c}_1^2 \max\left\{ \inv{(\abs{i-k}+\abs{j-l}+1)^d}, \inv{(\abs{i-l}+\abs{j-k} + 1)^d}\right\}, \quad \forall i,j,k,l \in \zn.
\end{equation}
In this paper we fix an arbitrary  $\alpha\in(2,d)$ and consider it as a universal constant.  Assume that there is a constant $\mathsf{c}_2>0$, such that $H$ allows a decomposition 
\begin{equation}\label{def:c2}
    H=\mathsf{c}_2 X + Y,
\end{equation}
where $X$ is a GOE matrix independent from $Y$.

We say that a constant is universal if it only depends on $\mathsf{c}_1,\mathsf{c}_2,d$ and $\phi$. In this paper we denote $a\lesssim b$ if there is a universal constant $c>0$ such that $a\leq cb$. We also denote $a\sim1$ if $a\lesssim 1$ and $1\lesssim a$.

For $\beta>0$ and any matrix $A$ (finite square or infinite) we define the following norms,
\begin{equation}
    \norm{A}_\beta : = \sup_{i,j} \left(\abs{A_{ij}}(1+\abs{i-j})^\beta\right),\quad \abs{A}_\infty : = \max_{i,j} \abs{A_{ij}}.
\end{equation}

Let $\lambda_1\leq \cdots \leq \lambda_N$ be the eigenvalues of $H$. Let $\hat{\lambda}_1\leq \cdots\leq \gamma_N$ be the eigenvalues of an $N$ by $N$ GOE matrix (i.e. a matrix $A=(Z_{ij}+Z_{ji})_{1\leq i,j\leq N}$ where $(Z_{ij})$ are i.i.d. copies of an $N(0,1/N)$ random variable). The main result we will prove is the following.
\begin{theorem}

There exists a universal constant $\gamma$  such that for any $f\in C^1( \mathbb{R}^{k-1})$, the following inequality holds for $N$ large enough.
\begin{gather}
 |\mathbb{E}_{H}[f(\gamma N^{2/3}(\lambda_{2} -\lambda_{1}), ...\gamma N^{2/3}(\lambda_{k} - \lambda_1))] - 
\mathbb{E}_{\text{GOE}}[f(N^{2/3}(\hat{\lambda}_{2}-\hat{\lambda}_1),...N^{2/3}(\hat{\lambda}_{k} -\hat{\lambda}_1))] \le N^{-c}
\end{gather}
\end{theorem}

\subsection{The Loop Equation}\label{sec:loopequation}
We will use the following lemma.
\begin{lemma}\label{lem:stein}
Let $Z=(Z_k)_{k=1}^p$ be a centered Gaussian random vector in $\bR^p$ with covariance matrix $\Sigma\in\bR^{p\times p}$. Let $f\in C^1(\bR^p)$.Then,
\[
    \EE{f(Z)Z_l} = \sum_{k=1}^p \EE{\partial_kf(Z)}\Sigma_{kl}, \forall 1\leq l\leq p.
\]
\end{lemma}
\begin{proof}
By a linear change of variable, we may assume without loss of generality that $\Sigma=I$ and $l=1$. It is sufficient to show that $\bE[f(Z)Z_1]=\bE[\partial_1f(Z)]$. Let $\mathcal{F}:=\sigma(Z_2,\cdots,Z_p)$, it is sufficient to show $\bE[f(Z)Z_1\vert \mathcal{F}] = \bE[\partial_1f(Z)\vert \mathcal{F}]$. This directly follows from an identity known as Stein's lemma, which says that if $X\sim N(0,1)$ and $h\in C^1(\bR)$, then $\EE{h(X)X}=\EE{h'(X)}$.
\end{proof}
We also use the following decomposition lemma.
\begin{lemma}\label{lem:decomp}
Let $Z=(Z_k)_{k=1}^p$ be a centered Gaussian random vector. Let $1\leq q< p$. Then, there is a constant matrix $(a_{kl})_{1\leq l\leq q,q+1\leq k\leq p}$ such that
\[
    Z_k = \sum_{l=1}^q a_{kl}Z_l + \tilde{Z}_k,
\]
where $(\tilde{Z}_k)_{k=q+1}^p$ are Gaussian random variables independent from $(Z_l)_{l=1}^q$.
\end{lemma}
\begin{proof}
Up to a linear transform, we may assume without loss of generality that $(Z_l)_{l=1}^q$ has covariance matrix $I_{q\times q}$. Let $\tilde{Z}_k:= Z_k - \sum_{l=1}^1 \EE{Z_kZ_l}Z_l$. It is easy to check that $(\tilde{Z}_k)_{k=q+1}^p$ are uncorrelated with $(Z_l)_{l=1}^q$. Being linear combinations of Gaussian random variables, $(\tilde{Z}_k)_{k=q+1}^p$ are still Gaussian.  Therefore $\tilde{Z}_k$ are independent from $(Z_l)_{l=1}^q$, since zero correlation is equivalent to independence for Gaussian random variables.
\end{proof}

We start with the trivial matrix identity $G(H-z) = I$, which can be written as follows
\begin{gather} \label{r:loopequation}
\sum_{k} G_{ik} h_{kj} - z G_{ij} = \delta_{ij}, \quad i,j\in\zn.
\end{gather}
Without loss of generality, fix $j=1$. According to Lemma \ref{lem:decomp} we may write, 
\begin{equation}\label{r:decomp}
	h_{ab} = \sum_{k=1}^N \gamma_{abk1} h_{k1} + \tilde h_{ab},
\end{equation}
 where $\tilde h_{ab}$ is a Gaussian random variable that is independent from $(h_{k1})_{1\leq k \leq N}$. In particular, $\gamma_{a1k1} = \delta_{ak}$, $\tilde{h}_{a1} = 0$, $\forall a\in\zn$. In order to apply Lemma \ref{lem:stein} on \eqref{r:loopequation}, let  $\mathcal{F}_{1}$ be the $\sigma$-algebra generated by $(\tilde h_{ab})_{a\neq 1,b\neq 1}$.  Define conditional expectation operator
\[
	\bE_1 [ \cdot ] : = \EE{ \cdot \vert \mathcal{F}_1}.
\]
We will then be able to apply Lemma \ref{lem:stein} to get the following
\begin{gather}\label{r:condexp1}
	 \delta_{i1} = \sum_k \bE_1[G_{ik} h_{k1}] - z \bE_1[G_{i1}] = -\sum_{k,a,b}\bE_1[ G_{ia}G_{bk}] \xi_{abk1} - z \bE_1[G_{i1}] .
\end{gather}
For technical reasons define the cut-off version of $\xi$ as follows, $\tilde{\xi}_{iklj}= \min\{\max\{\xi_{iklj},-\mathsf{c}_1^2\abs{i-j}^{-d}\},\mathsf{c}_1^2\abs{i-j}^{-d}\}$, so that $\tilde{\xi}_{iklj}$ has a power-law decay as $i$ and $j$ gets farther. Define a linear map $S:\bR^{N\times N} \to \bR^{N\times N}$ by.
\begin{equation}\label{def:S}
    	(S(M))_{pq} : = \inv{N}\sum_{\alpha,\beta} \tilde{\xi}_{p\alpha\beta q}  M_{\alpha\beta}.
\end{equation}
Therefore, \eqref{r:condexp1} is equivalent to
\begin{equation}\label{E1}
	 - \bE_1 \left[[G S(G)]_{i1}\right]- z \bE_1\left[G_{i1} \right] = \delta_{i1} + O(N^{-1}\max_{k,l}\abs{G_{kl}}).
\end{equation}
Notice that the expectation operator $\bE_1$ is equivalent to integrating over $N$ weakly dependent Gaussian random variables, we may remove the expectation up to the cost of some small error terms, after which, we would get a self-consistent equation in the following form.
\begin{equation}
    G(-S(G)-z) = I + \textrm{error}.
\end{equation}
Define a map $F:\bR^{N\times N} \to \bR^{N\times N}$ via
\begin{equation}\label{def:F}
    F(M) = (-z- S(M))^{-1}.
\end{equation}
Then the above equation can be written as the perturbation of a fixed point equation
\begin{equation}
    G=F(G) +  \textrm{error}.
\end{equation}
Here the error is entry-wise bounded by roughly $O((N\eta)^{-\inv{2}})$. However, this entry-wise bound is not strong enough to use the stability of the equation $G=F(G)$. Therefore, we iterate the map $F$ on $G$ to get 
\begin{equation}
    F(F(G)) = F(F(F(G))) + \textrm{new error}.
\end{equation}
The new error term has a power-law decay on the off-diagonal entries, hence is much smaller than the original error. This allows us to get an estimate on $F(F(G))$. Using $F(F(G))$ we can recover $G$ and get a bound on $\abs{G-G_0}$ where $G_0$ is some deterministic matrix.


\subsection{Limiting Version of self-consistent equation}
Consider $\cK:=C(\bT^2)$ and $\cK_+ :=\{ g \in\cK\vert \Im g(s,u)>0,\forall s,u\in\bT\}$. Recall the function $\phi$ in \eqref{def:xi}. Define 
\begin{equation}
    \varphi(s,t,u,v):= \sum_{k,l}\phi(s,t,k,l) \re^{-2\pi\ri(uk-vl)}.
\end{equation}
The argument in Lemma 4.15 of \cite{Che2016} can be modified to show that $\varphi\sim 1$. Also, the decay condition \eqref{def:xidecay} guarantees that $\varphi$ is Lipschitz. Define $\Psi:\cK_+\to\cK_+$ via $\Psi(h)(s,u):=\iint_{\bT^2}\varphi(s,t,u,v)h(t,v)\dd t\dd v$ and $\Phi:\cK_+\to\cK_+$ via $\Phi(h):=(-\Psi(h)-z)^{-1}$. Consider the fixed point equation $g = \Phi(g)$, or equivalently,
\begin{equation}\label{sceinfty}
    g(-\Psi(g)-z)=1.
\end{equation}
If we think of $\hat{g}$ as an infinite matrix, we may write the above equation as 
\begin{equation}\label{sceinfty1}
    \hat{g}(-S(\hat{g})-z)=I.
\end{equation}
Equations like \eqref{sceinfty} are studied in detail in \cite{Ajanki2015a}. Since the function $\varphi$ is bounded above and below away from $0$, the function $\Phi$ satisfies conditions A1-A3 and is block fully indecomposible in Definition 2.9 of \cite{Ajanki2015a}. Also, since $\varphi$ is Lipschitz, it satisfies (2.22) in that article. Therefore,  their Theorem 2.6 says that the above equation has a unique solution $g\in\cK_+$, and there is a universal constant $\mathsf{c}_3<+\infty $ such that 
\begin{equation}
    \sup_{z\in\bC^+} \norm{g}_\infty \leq \mathsf{c}_3.
\end{equation}
Let $m(z): = \iint_{\bT^2} {g}(s,u)\dd s \dd u$. Then $m$ is the Stieltjes transform of a compactly supported probability measure $\nu$ on $\bR$, i.e.,
\begin{equation}
    m(z) = \int_\bR \frac{\nu(\dd x)}{x- z},\quad\forall z\in\bC^+.
\end{equation}
Then Theorem 2.6 in \cite{Ajanki2015a} says that $\nu$ has a $\frac{1}{3}$-H{\"o}lder continuous density $\rho \in C_c(\bR)$ such that it has square-root behavior at the left and right edges, i.e., let 
\begin{equation}
    E_L:= \inf \supp \nu, \quad E_R:= \sup \supp \nu.
\end{equation}
Then, there are $c_L, c_R>0$ s.t.
\begin{equation}
    \rho(E_L + t) = c_L\sqrt{t} + O(t),\quad \rho(E_R-t) = c_R\sqrt{t}+O(t),\text{ as } t\to 0_+.
\end{equation}
For $h\in\cK$, define the Fourier coefficients $\hat{h}(s,k):= \int_\bT h(s,u) \re^{-2\pi\ri k u} \dd u$. On $\cK$ we may define a norm $\norm{\cdot}_\beta$ for $\beta\geq 0$:
\begin{equation}
    \norm{h}_\beta:= \sup_{s\in\bT,k\in\bZ} \abs{\hat{h}(s,k)}(1+\abs{k})^\beta.
\end{equation}
In view of Theorem \ref{Jaffard}, it is easy to see that $\norm{g}_\alpha\vee \norm{g^{-1}}_\alpha \lesssim 1$ on any bounded subdomain of $\bC^+$. For any $N\in\mathbb{N}$, define a discretization operator $D^{(N)}: \cK \to \cnn$ by
\begin{equation}
    D(h)_{ij}: = \hat{h}(i/N, j-i).
\end{equation}
We have the following lemma concerning the discretization $D(g)$:
\begin{lemma}\label{D}
Let $a,b\in\cK$. Assume that $b$ is Lipschitz in the first variable in the sense that $\abs{b(s,u)-b(s',u)}\leq L\abs{s-s'},\forall s,s'\in\bT,u\in\bT$. Then, $\norm{D(a)D(b)-D(ab)}\lesssim N^{-\inv{2}}(L+\norm{b}_\alpha)\norm{a}_\alpha$, also, $\norm{D(a)D(b)^*-D(a\bar{b})}\lesssim N^{-\inv{2}}(L+\norm{b}_\alpha)\norm{a}_\alpha$.
\end{lemma}
\begin{proof}
By definition, $(D(a)D(b)-D(ab))_{ij}= \sum_k \hat{a}(i/N,k-i)(\hat{b}(k/N,j-k)-\hat{b}(i/N,j-k))$, therefore, using the decay of $\hat{a}$ and the Liptchitz continuity of $\hat{b}$, we have 
\begin{equation}\label{D1}
    \abs{(D(a)D(b)-D(ab))_{ij}} \leq \sum_k \frac{\norm{a}_\alpha}{\abs{k-i}^\alpha}\frac{L\abs{k-i}}{N}\lesssim N^{-1}L\norm{a}_\alpha.    
\end{equation}
On the other hand, $\norm{D(a)D(b)-D(ab)}_\alpha\lesssim \norm{a}_\alpha \norm{b}_\alpha$, hence
\begin{equation}\label{D2}
        \abs{(D(a)D(b)-D(ab))_{ij}} \leq \norm{a}_\alpha\norm{b}_\alpha(1+\abs{i-j})^{-\alpha}.
\end{equation}
Therefore $\norm{D(a)D(b)-D(ab)}_{l^\infty\to l^\infty} \lesssim (L+\norm{b}_\alpha)\norm{a}_\alpha\sum_k(N^{-1}\wedge  \abs{k}^{-2} ) \lesssim N^{-\inv{2}}(L+\norm{b}_\alpha)\norm{a}_\alpha$. Similarly, the $l^1\to l^1$ norm is bounded by the same quantity, hence the operator norm has the same bound by interpolation. The second estimate follows from a similar argument.
\end{proof}

Let $Z(z):= \{ \abs{g(s,t)}\vert s,t\in\bT\}$. From equation \eqref{sceinfty} we know that $Z$ is bounded away from $0$ and $+\infty$. For $K>0$ let $\mathcal{D}_K=\{z\in\bC^+\vert \abs{z}\leq K\}$.
\begin{corollary} \label{col:DisSV}
There is an $N(K)>0$ such that for any $ N>N(K)$ and $z\in \mathcal{D}_K$, the singular spectrum of $D(g)$ is in the $N^{-\inv{3}}(\log N)^{-1}$-neighborhood of $Z(z)$.
\end{corollary}
\begin{proof}
Let $\theta\ll 1 $ be some parameter to be chosen. Let $x\in\bR_+$ s.t. $\dist(x,Z)\geq \theta$. Let $h:= \frac{1}{\abs{g}^2 -x^2}$. Then 
\[
    \norm{(D(g)D(g)^* -x^2) D(h) -I}\leq \norm{D(g)D(g)^*-D(\abs{g}^2)}\norm{D(h)} + \norm{D(\abs{g}^2-x^2)D(h)-D(1)}.
\]
According to Lemma \ref{D}, we have $\norm{(D(g)D(g)^* -x^2) D(h) -I}\lesssim \norm{h}_2+L$, where $L$ is the Lipschitz constant of $h$ with respect to the first variable. By chain rule we know that $\norm{h''}_\infty \lesssim \theta^{-3}$ and $L\lesssim \theta^{-2}$. Therefore, $\norm{h}_2\lesssim \theta^{-3}$ and hence $\norm{(D(g)D(g)^* -x^2) D(h) -I}\lesssim \theta^{-3}$. Choose $\theta = N^{-\inv{3}}(\log N)^{-1}$. Then $D(g)D(g)^*-x^2$ is invertible for $N$ large enough. That means $x$ is not in the singular spectrum of $D(g)$.
\end{proof}

\begin{corollary}\label{col:DisSV1}
Let $R:= D(g)(-S(D(g))-z)-I$. Then, for any $z\in \mathcal{D}_K$,
\begin{equation}
    \abs{R_{ij}}\leq C(K) N^{-1} \wedge \abs{i-j}^{-2}.
\end{equation}
In particular, $\norm{R}\leq C(K) N^{-\inv{2}}$.
\end{corollary}
\begin{proof}
According to Lemma \ref{D} and equation \eqref{sceinfty}, we know
\[
    \abs{(D(g)(-D(\Psi(g))-z)-I)_{ij}}\lesssim N^{-1} \wedge \abs{i-j}^{-2}.
\]
By definition, $(D(\Psi(g))_{kl} = \sum_p \int_\bT \phi(k/N,t,l-k,p) \hat{g}(t,p)\dd t$, $(S(D(g)))_{kl}=\inv{N}\sum_{p,q} \phi(k/N,q/N,l-k,p)\hat{g}(t,p)$. Using the Lipschitz-ness of $\phi$ and $g$, we have $\abs{(D(\Psi(g))_{kl}-(S(D(g)))_{kl})}\lesssim N^{-1} \wedge \abs{k-l}^{-2}$. Therefore,
\[
    \abs{(D(g)(-S(D(g))-z)-I)_{ij}}\lesssim N^{-1} \wedge \abs{i-j}^{-2},
\]
as desired.
\end{proof}

\begin{corollary} \label{col:Nrmbndit}
Recall the definition \eqref{def:F} of $F$. For all sufficiently large N, there exists a constant $c>0$ such that $\norm{F(D(g))-D(g)}\vee\norm{F(F(D(g)))-D(g)}\leq cN^{-\inv{2}}$. In particular, the singular spectrum of $F(D(g))$ and $F(F(D(g)))$ are contained in a compact subset of $\bR_+$.
\end{corollary}
\begin{proof}
Using the notation from the previous corollary, if $D(g)(-S(D(g))-z)-I=R$, then 
\[
    F(D(g)) = (I+R)^{-1}D(g).
\]
Since $\norm{R}\lesssim N^{-\inv{2}}$ and $\norm{D(g)}\lesssim 1$, we know $\norm{(I+R)^{-1}D(g)-D(g)}\lesssim N^{-\inv{2}}$. From perturbation theory we know that the singular spectrum of $F(D(g))$ is within the $N^{-\inv{2}}$ of that of $D(g)$, therefore it is a compact subset of $\bR_+$. On the other hand, a simple algebraic calculation yields
\[
    F(F(D(g)))= \left(I+F(D(g))S(F(D(g))R)\right)^{-1}F(D(g)).
\]
Note that $\norm{F(D(g))S(F(D(g))R)}\lesssim N^{-\inv{2}}$, so the singular spectrum of $F(F(D(g)))$ is within the $O(N^{-\inv{2}}$ neighborhood of that of $F(F(D))$, hence is a compact subset of $\bR_+$.
\end{proof}

For $z\in\bC^+$, define 
\begin{equation}\label{def:kappa}
    \kappa(z):= \dist(z,\supp \nu),\quad \rho(z):= \rho(\Re z),\quad \omega(z):=\kappa(z)^{\frac{2}{3}} + \rho(z)^2.
\end{equation}
Theorem 2.8 in \cite{Ajanki2015a} implies the following stability result:
\begin{lemma}\label{lem:stab}
    There is a universal constant $\mathsf{c}_6$ such that if $\tilde{g} \in \cK$ satisfies
    \begin{equation}
        \tilde{g}(-{\Psi}(\tilde{g}) -z) = 1 + r
    \end{equation}
    and $\norm{\tilde{g}-g}_\infty\leq \mathsf{c}_6 ( \kappa^{\frac{2}{3}} + \rho)$, then $\norm{\tilde{g}-g}_\infty\leq \mathsf{c}_6^{-1} \omega^{-1}$.
\end{lemma}
\subsection{Concentration lemmas}
The following lemma says that a Lipschitz function of weakly dependent Gaussian random variables concentrates around its expectation. 
\begin{lemma}\label{lem:gaussianconcentration}
Let $X = (X_1,\cdots,X_n)$ be an array of centered Gaussian random variables with covariance matrix $\Sigma$. Let $f:\bR^N\to\bR $ be a Lipschitz function, such that $\abs{f(x)-f(y)} \leq L\abs{x-y},\forall x,y\in\bR^N$. Then
\[
		\bP\left[ \left| f(X) - \bE f(X)\right|\geq t \right] \leq 2 \re ^{-\frac{t^2}{2L^2 \norm{\Sigma}}} ,\quad\forall t> 0.
\]
\end{lemma}
\begin{proof}
Let $Y=\Sigma^{-1/2} X $ so that $Y$ is an $n$-dimensional random vector with independent $N(0,1)$ components. In \cite{Wainwright},
\[
	\bP\left[ \left| f(\Sigma^{\frac{1}{2}}Y ) - \bE f(\Sigma^{\frac{1}{2}}Y ) \right|\geq t \right] \leq 2 \re ^{-\frac{t^2}{2L_1^2}} \text{ for all } t> 0.
\]
Here $L_1$ is the Lipschitz constant for the function $y\mapsto f(\Sigma^{\frac{1}{2}}y)$. It is easy to see that $L_1 \leq L \norm{\Sigma}^{\frac{1}{2}}$, which concludes the proof.
\end{proof}
In the future, we will frequently use the following lemma.
\begin{lemma}\label{lem:decaymatrix}
Let $A\in\bC^{N\times N}$. Assume that there are $\beta>0, \theta >1$,  s.t. $\abs{A_{ij}}\leq \beta(\left( \abs{i-j} + 1\right)^{-\theta} + N^{-1}).\forall 1\leq i,j \leq n$. Then $ \norm{A} \leq \frac{\beta\theta}{\theta-1}$. More generally, for any $p\in[1,+\infty]$, we have $\norm{A}_{l^p\to l^p} \leq \frac{\beta\theta}{\theta-1}$.
\end{lemma}
\begin{proof}
	Without loss of generality let $\beta=1$. For any vector $v\in\bR^n$,
	\[
		\norm{A v}_\infty = \max_k\abs{\sum_i A_{ki}v_i}\leq \norm{v}_\infty \max_k \left(\sum_i ((\abs{i-k}+1)^{-\theta} + N^{-1})\right) \leq\norm{v}_\infty \left(\int_{1}^{+\infty}x^{-\theta}\dd x + 1\right).
	\]
	Therefore, $\norm{A}_{l^\infty \to l^\infty} \leq \frac{\beta\theta}{\theta-1}$. Similarly, $\norm{A}_{l^1\to l^1} = \norm{A^*}_{l^\infty\to l^\infty} \leq \frac{\beta\theta}{\theta-1}$. By interpolation,
	\[
		\norm{A}_{l^p\to l^p}\leq \norm{A}_{l^\infty \to l^\infty}^ { \inv{p}}\norm{A}_{l^1\to l^1}^{1-\inv{p}} \leq \frac{\beta\theta }{\theta-1},\quad \forall p\in[1,+\infty].
	\]
\end{proof}

Recall that in Section \ref{sec:loopequation} we defined a map $S$ (see \eqref{def:S}). Thanks to the decay condition \eqref{def:xidecay}, the operator $S$ is a bounded operator, as will be seen in the following lemma.
\begin{lemma}\label{lem:S}
Let $A\in\bC^{N\times N} $. Then there is a universal constant $c>0$ such that the following inequalities hold.
\begin{enumerate}
    \item $\norm{S(A)}_{d-1} \leq c \abs{A}_\infty$.
    \item $\norm{S(A)}_{l^p\to l^p}\leq c\abs{A}_\infty, \forall p\in[1,+\infty]$.
    \item $\norm{S(A)}_d \leq c\norm{A}_{d-1}$.
    \item $\norm{S(A)}_{d-\inv{2}}\leq c\norm{A}.$
\end{enumerate}
\end{lemma}
\begin{proof}
By definition $\abs{S(A)_{ij} }= \abs{\inv{N} \sum_{k,l} \xi_{iklj} A_{kl}}\leq \frac{\abs{A}_\infty}{N}\sum_{k,l} \abs{\xi_{iklj}} $. According to \eqref{def:xidecay}, $\inv{N}\sum_{k,l} \abs{\xi_{iklj}}\lesssim \inv{(1+\abs{i-j})^{d-1}} $. Hence $\abs{S(A)_{ij} } \lesssim  \frac{\abs{A}_\infty}{(1+\abs{i-j})^{d-1}} $, which implies the first inequality.  Setting $\theta=d-1$ in Lemma \ref{lem:decaymatrix}, we see that 
$\norm{S(A)}_{l^p\to l^p} \lesssim  \abs{A}_\infty, \forall p\in[1,+\infty]$, which implies the second inequality. If $\norm{A}_{d-1}<+\infty$, then $\abs{S(A)_{ij} }= \abs{\inv{N} \sum_{k,l} \xi_{iklj} A_{kl}}\leq (1+\abs{i-j})^{-d}\frac{1}{N}\sum_{k,l}\abs{A_{kl}}\lesssim (1+\abs{i-j})^{-d}$. This proves the third inequality. As for the fourth inequality, we use Cauchy-Schwarz inequality to see that 
$ \abs{S(A)_{ij} }\leq \inv{N} \left(\sum_{k,l}\abs{\xi_{iklj}}^2\right)^{\inv{2}} \left(\sum_{k,l}\abs{A_{kl}}^2\right)^{\inv{2}} \lesssim (1+\abs{i-j})^{\inv{2}-d}\norm{A}$.
\end{proof}

\subsection{Error estimate}\label{sec:errors}
Recall the decomposition \eqref{r:decomp}
\begin{equation}\label{r:decomp1}
	h_{ab} = \sum_{k=1}^N \gamma_{abk1} h_{k1} + \tilde h_{ab}, \quad \forall a,b\in\zn.
\end{equation}
Taking the co-variance with $h_{l1}$ for any $l\in\zn$, we see that 
\[
    \xi_{abl1} = \sum_{k=1}^N \gamma_{abk1} \xi_{l1k1},\quad \forall l\in\zn.
\]
Note that by assumption \eqref{def:xidecay} the matrix $\Sigma_1:= (\xi_{l1k1})_{l,k\in\zn}$ satisfies $\abs{\xi_{l1k1}}\lesssim \frac{1}{(1+\abs{l-k})^\alpha}$ and by \eqref{def:c2}, $\norm{\Sigma_1^{-1}} \leq \mathsf{c}_2^{-1}$. Therefore Lemma \ref{Jaffard} implies that $\abs{(\Sigma^{-1}_1)_{ij}} \lesssim (1+\abs{i-j})^{-\alpha}$ and hence by Lemma \ref{lem:decaymatrix} we have $\norm{\Sigma_1^{-1}} \leq c$. Let $\nabla_1$ denote the partial gradient with respect to the first column $(h_{k1})_{1\leq k\leq N}$. Use the fact that $\frac{\partial G_{i1}}{\partial h_{ab}} = - G_{ia} G_{bj}$ and the chain rule, we have 
\begin{equation}\nonumber
\norm{\nabla_1 G_{ij}}^2 \leq \sum_k \abs{-\sum_{a,b} G_{ia}G_{bj}\gamma_{abk1}}^2\lesssim \sum_k \abs{-\sum_{a,b} G_{ia}G_{bj}\xi_{abk1}}^2.
\end{equation}
In the second inequality above we have used the boundedness of $\norm{\Sigma_1^{-1}}$. Let 
\begin{equation}\label{def:gamma}
    \Gamma = \max_{i,j}\abs{G_{ij}}\vee 1,\quad \gamma:=\max_{i} \Im G_{ii} \vee \eta.
\end{equation}
Use the decay rate \eqref{def:xidecay}, 
    \begin{equation}
        \norm{\nabla_1 G_{ij}}^2 \leq C\Gamma^2 \sum_k  \left( \sum_a \frac{\abs{G_{ia}}^2}{(\abs{a-k}+1)^{\alpha-1}} + \sum_b \frac{\abs{G_{bj}}^2}{(\abs{b-k}+1)^{\alpha-1}}\right)^2.
    \end{equation}
Since $\alpha-1 >1$, the operator norm of the matrix $\left( \frac{1}{(|a-k| + 1 )^{\alpha-1}}\right) _ { 1\leq a, k \leq N}$ is bounded by $C(\alpha-2)^{-1}$, according to Lemma \ref{lem:decaymatrix}. Therefore, 
\begin{equation}\label{nablag}
\norm{\nabla_1 G_{ij}}^2 \leq C\Gamma^2 \left( \sum_a \abs{G_{ia}}^2 +\sum_b \abs{G_{bj}}^2\right) \leq C\Gamma^2 \gamma\eta^{-1}.
\end{equation}
In the second inequality we used Ward Identity.
Similarly,
\begin{equation}
    \norm{\nabla_1 (GS(G))_{ij}}\leq \norm{\sum_p \nabla_1 G_{ip} (S(G))_{pj}} + \norm{\sum_p  G_{ip} \nabla_1(S(G))_{pj}}.
\end{equation}
Define a short-hand notation $Q_{kl}:= \norm{\nabla_1 G_{kl}}$. By \eqref{nablag} we have $\abs{Q}_\infty^2 \leq C\Gamma^2 \gamma \eta^{-1}$. Then
\begin{equation}
\begin{aligned}
      \norm{\nabla_1 (GS(G))_{ij}} &\leq \sum_p Q_{ip} \abs{(S(G))_{pj}} + \sum_p \abs{G_{ip}}\inv{N}\sum_{k,l}\abs{\xi_{pklj}}Q_{kl}\\
        &\leq \abs{Q}_\infty \norm{S(G)}_{l^\infty \to l^\infty } + \abs{Q}_\infty \frac{\Gamma}{N}\sum_{k,l,p}\abs{\xi_{pklj}}.
\end{aligned}
\end{equation}
Now we use the bound \eqref{nablag}, and use the decay \eqref{def:xidecay} as well as Lemma \ref{lem:S} to see,
\begin{equation}\label{nablagsg}
    \norm{\nabla_1 (GS(G))_{ij}}^2 \leq C\Gamma^4\gamma\eta^{-1}.
\end{equation}

The observation above yields the following lemma.
\begin{lemma} \label{lem:ibp} Let $z=E+\ri\eta\in\bC^+$ and $K\geq 1$, then there is a universal constant $c>0$ such that 
    \[
        -GS(G)-Gz = I + R,
    \]
    where $
	\PP{\abs{R}_\infty \geq t \sqrt{\frac{K^4\gamma}{N\eta}} , \Gamma \leq K}\leq 2N^2\re^{-ct^2},\forall t\geq 1$. 
\end{lemma}
\begin{proof}
	For any $K>0$ let $\chi:\bR \to [0,1]$ be a smooth function s.t. $\abs{\chi'} \leq 1$ and $\chi = 1$ on $[-K, K]$ and $\chi=0$ outside $[-3K,3K]$. Define 
	\[
	\tilde{G} = \chi(\Gamma) G.
	\]
	Then
	$\norm{	\nabla \tilde{G}_{ij}}^2\lesssim K^2\gamma\eta^{-1}.$ According to Lemma \ref{lem:gaussianconcentration},
	\[
	\PP{\abs{ \tilde G_{ij} - \bE_j\tilde {G}_{ij}} \geq t \sqrt{\frac{K^2\gamma}{N\eta}}}\leq 2\re^{-ct^2}.
	\]
	Note that $\tilde{G}=G$ on the event  $\{ \Gamma\leq K \}$. Therefore, 
	\[
	\PP{\max_{i,j}\abs{ G_{ij} - \bE_j{G}_{ij}} \geq t \sqrt{\frac{K^2\gamma}{N\eta}}, \Gamma \leq K}\leq 2N^2\re^{-ct^2}.
	\]
    On the other hand, in view of \eqref{nablagsg}, a similar argument yields,
	\[
	\PP{\max_{i,j}\abs{ (GS(G))_{ij}- \bE_j(GS(G))_{ij}} \geq t \sqrt{\frac{K^4\gamma}{N\eta}} , \Gamma \leq K}\leq 2N^2\re^{-ct^2} .
	\]
    Now we go back to the identity \eqref{E1}, removing $\bE_1$ at the cost of some error term, and replacing $1$ with a generic $j$, to see
    \[
        -GS(G)-Gz = I + R,
    \]
    where $
	\PP{\abs{R}_\infty \geq t \sqrt{\frac{K^4\gamma}{N\eta}} , \Gamma \leq K}\leq 2N^2\re^{-ct^2}$. 
\end{proof}
In particular, for a crude bound, we may take $t= \log N$ and take $K = 2/\eta$ so that $\PP{\Gamma>K} = 0$. The lemma above yields,
\begin{corollary}\label{cor:globalerror}
Let $R$ satisfy
\[
    G(-S(G)-z) = I + R.
\]
Then $\abs{R}_\infty \leq \frac{8\log N }{\sqrt{N\eta^{6}}}$ with probability $1-N^{-c\log N}$.
\end{corollary}

\section{The Local Law for Correlated Gaussian Ensembles}

\subsection{Power Law Decay of Inverse Matrices}

\begin{lemma}\label{ProdofDecay}
Let $A,B\in\cnn$, $\beta_{1,2}>1$, then $\norm{AB}_{\min\{\beta_1,\beta_2\}} \leq C_{\min{\beta_1,\beta_2}}\norm{A}_{\beta_1}\norm{B}_{\beta_2}$.
\end{lemma}
\begin{proof}
Note that by definition, $\norm{A}_{\min\{\beta_1,\beta_2\}}\norm{B}_{\min\{\beta_1,\beta_2\}}\leq \norm{A}_{\beta_1}\norm{B}_{\beta_2}$, so it is sufficient to prove the case where $\beta_1=\beta_2=\beta$. Without loss of generality assume $\norm{A}_\beta=\norm{B}_\beta=1$, then,
\[
    \abs{(AB)_{ik}} \leq \sum_{j} \frac{1}{(1+\abs{i-j})^\beta} \frac{1}{(1+\abs{j-k})^\beta}.
\]
Since either $\abs{i-j}$ or $\abs{j-k}$ is $\geq \abs{i-k}/2$, the above quantity is bounded by
\[
     \abs{(AB)_{ik}}\leq 2 \sum_{l\in\bZ} \frac{1}{(1+\frac{\abs{i-k}}{2})^\beta} \frac{1}{(1+\abs{l})^\beta} \leq \frac{2}{(1+\frac{\abs{i-k}}{2})^\beta} \left( 1+ 2\int_{1}^{+\infty} \frac{\dd x}{x^\beta}\right),
\]
which is bounded by $2^{\beta +1} \frac{\beta+1}{\beta-1} (1+\abs{i-k})^{-\beta}$.
\end{proof}

The following argument is based off a similar argument of Jaffard \cite{Jaffard}. 
\begin{theorem}\label{Jaffard}
Let $d>\frac{3}{2}$ and assume that a matrix $A=I+B$ (finite or infinite) satisfies $\norm{B}<1$ and $\norm{A}_d <+\infty$. Then, for any $\delta>0$, there exists a polynomial dependent on $d$ and $\delta>0$ such that $\norm{A^{-1}}_{d-1/2-\delta}\leq P_{d,\delta}(\norm{A}_d,\frac{1}{1-\norm{B}})$.

If $d>1$ and there exists an $\epsilon>0$ such that $\norm{B} \le 1- \epsilon$, then  $\norm{A^{-1}}_{d-\delta} \le C(\delta,\epsilon, \norm{A}_d)$.
\end{theorem}

We will show matrix element decay of the solution to the self-consistent equation. Though we will only really apply this to the solution of the limiting equation \eqref{sceinfty1}, the following theorem will phrase the result in terms of Matrices for convenience of notation.

\begin{proposition}\label{DecayofSolution}
    Let $M$ be the solution to the following equation
    \begin{gather*}
        M(-z -S(M))=I.
    \end{gather*}
    If there exists a constant $c>0$ such that $\norm{M}, \norm{M^{-1}} \leq c$, then we have that $\norm{M}_{\alpha} \leq C(c,\alpha)$.
\end{proposition}
\begin{proof}
Notice that we are able to write
\begin{gather*}\label{ProdForm}
M=(M^{-1})^*((M^{-1})^*M^{-1})^{-1}.
\end{gather*}
By the equation of $M$, we have $M^{-1}=-z-S(M)$.  Let us first estimate the decay of $M^{-1}$. By Lemma \ref{lem:S} we have $\norm{M^{-1}}_{d-\inv{2}}\lesssim \norm{M}$. By Lemma \ref{lem:decaymatrix} we have $\norm{M^{-1}(M^{-1})^*}_{d-\inv{2}}\lesssim \norm{M}^2$. We would now like to apply theorem \ref{Jaffard}  to $(M^{-1}(M^{-1})^*)^{-1}$. 

For any general positive semi-definite matrix, $A$, we will be able to write it as $A = \frac{\lambda_1 + \lambda_n}{2}[I + B]$ where $\lambda_1$ and $\lambda_n$ are respectively the largest and smallest eigenvalues of A.

 Theorem \ref{Jaffard} is applied  to the matrix $I+B$. The operator norm bound on $B$ will be $\frac{\lambda_1 -\lambda_n}{\lambda_1 + \lambda_n}$. 
The important factor $r= 1- \|B\|$ will be $\frac{2 \lambda_n}{\lambda_1 + \lambda_n}$.  \ref{Jaffard}  now shows that the matrix decay of $A^{-1}$ will be the same matrix decay of $A^{-1}$.

Applying this logic to the positive semidefinite matrix $(-z-S(M))(-z-S(M))^*$, one will obtain that $\lambda_1$ and $\lambda_n$ are both of some bounded constant order. Thus, we see we have matrix decay of order $1$.

Finally applying the multiplication lemma \eqref{ProdofDecay} to \eqref{ProdForm}, we will be able to get a matrix decay of $M$ of order $1$. We use this decay of $M$ to argue that $-z-S(M)$ has a matrix decay of order $\alpha$. We can then apply the same logic as above to argue that $M$ has matrix decay of order $\alpha$.
\end{proof}

\begin{remark}
The solution to the limiting self-consistent equation, though ostensibly a vector, can be written as an infinite Toeplitz matrix and the above result can be applied.
\end{remark}

Now we define $J:\cnn\to \cK$, such that for any $A\in \cnn$ and $i\in\zn$, $u\in\bT$,
\begin{equation}\label{def:J}
    J(A)(i/N,u):= \sum_{k=i-\floor{N/2}}^{i+\floor{N/2}} A_{i,i+k}\re^{2\pi\ri ku}.
\end{equation}
and $J(A)(s,u)$ is linear in $s$ for $s\in[i/N,(i+1)/N]$. It is easy to check that 
\begin{equation}\label{DJ}
    D(J(A)) = A,\quad \forall A\in\bC^{N\times N}.
\end{equation}

\begin{proposition}\label{prop:iterate}
    Consider a fixed bounded subset $U\subset \bC^+$. There are constants $\epsilon, C>0$ such that if $\abs{J(M)-g}_\infty\leq \epsilon$ , then $\norm{F(M)-F(D(g))}_{\alpha-1}\vee \norm{F(F(M))-F(F(D(g)))}_\alpha \leq C\abs{J(M)-g}_\infty$ and $\abs{F(M)-D(g)}_\infty\vee \abs{F(F(M))-D(g)}_\infty\leq C(\abs{J(M)-g}_\infty+N^{-\inv{2}}),\forall z\in U$.
\end{proposition}
\begin{proof} Let $A:=F(D(g))$ and $R:=S(M-D(g))$. Then 
\[
    F(M)-A= \sum_{k=1}^\infty A(RA)^k.
\]
Hence $\norm{F(M)-A}_{\alpha-1}\leq \sum_{k=1}^\infty \norm{A(RA)^k}_{\alpha-1}$. It is easy to see that $\norm{R}_{\alpha-1}\leq  c\abs{J(M)-g}_\infty $ for some universal constant $c>0$. By Lemma  \ref{ProdofDecay} we have $\norm{A(RA)^k}_{\alpha-1}\lesssim (c\abs{J(M)-g}_\infty)^k$. Therefore, taking $\epsilon$ small enough, we have $\norm{F(M)-F(D(g))}_{\alpha-1} \leq C\abs{J(M)-g}_\infty,$.

Next, we define $R'=S(F(M)-F(D(g)))$, $A' = F(F(D(g)))$. Then $\norm{R'}_\alpha \leq c'\epsilon$ according to the above argument. We have
\[
    F(F(M))-A' = \sum_{k=1}^\infty A'(R'A')^k.
\]
By Lemma  \ref{ProdofDecay} we have $\norm{A'(R'A')^k}_\alpha \lesssim (c\abs{J(M)-g}_\infty)^k$. Therefore, taking $\epsilon$ small enough, we have $\norm{F(F(M))-F(F(D(g)))}_\alpha \leq C\abs{J(M)-g}_\infty$.

The last claim follows from the estimates above and Corollary \ref{col:Nrmbndit}.
\end{proof}

\subsection{Local Law}


Recall definition \eqref{def:kappa} and \eqref{def:gamma}, for a constant $T>0$ to be chosen, define 
\begin{equation}
    \mathcal{D} : = \{z\in \bC^+\vert \abs{z}\leq T, \Im z \geq (\log N)^{10}N^{-1} \omega^{-4}\}.
\end{equation}
\begin{theorem}[Local law]\label{thm:bulklaw}
Define $\Lambda(z):=\abs{D(g)-G}_\infty$. For $N$ large enough, we have
\[
    \sup_{z\in \mathcal{D}} \Lambda(z) \leq (\log N)^4 \left(\sqrt{\frac{\gamma}{N\eta}}\right)\omega^{-1},
\]
with probability $1- \re^{-\mathsf{a}_3(\log N)^2}$. If $\kappa>\rho$,
\[
    \sup_{z\in \mathcal{D}^\theta_N}\Lambda(z) \leq (\log N)^8 \left(\sqrt{\frac{\Im m}{N\eta}}\omega^{-1}+(N\eta)^{-1}\omega^{-2}\right) ,
\]
with probability $1- \re^{-\mathsf{a}_3(\log N)^2}$.
\end{theorem}

\begin{proof}
Take $K:=\log N$, and let $\{z_k\}$ be an $N^{-4}$-net of $\mathcal{D}$. Define 
\[
    \Omega:= \bigcup_{k=1}^{N^{10}}\{ \Lambda(z)\in (K^4\sqrt{\gamma}(N\eta)^{-\inv{2}}\omega^{-1}, K^{-1}\omega\}
\]
Then by Proposition \ref{prop:iterate}, on $\Omega$ we have 
\[
    F(F(G))(-S(F(F(G))-z) = I + \tilde{R},
\]
where $\tilde{R}\lesssim \abs{R}_\infty$. Then $J(F(F(G))(-\Psi(J(F(F(G))))-z) = 1 + O(\abs{R}_\infty+N^{-1})$. By Proposition \ref{prop:iterate}, on $\Omega$ we have $\abs{F(F(G))-D(g)}_\infty\lesssim K^{-1} \omega  + N^{-\inv{2}}$, which is $\ll \omega$. By stability Lemma \ref{lem:stab} we know $\Lambda(z) \lesssim (\abs{R}_\infty+N^{-1})\omega^{-1} $, which implies $\norm{J(F(F(G)))-g}_\infty\lesssim (\abs{R}_\infty+N^{-1})\omega^{-1}$, hence $\abs{G-D(g)}_\infty \lesssim (\abs{R}_\infty+N^{-1})\omega^{-1}$. Therefore, on $\Omega$ we have $\abs{R}_\infty \gtrsim K^4\sqrt{\gamma}(N\eta)^{-\frac{1}{2}}$. By Lemma \ref{lem:ibp} we know $\PP{\Omega}\leq 2N^{12}\re^{-c(\log N)^2}$. On $\Omega^c$, we either have $\inf_{z\in\mathcal{D}}\abs{G-D(g)}_\infty\geq K^{-1} \omega /2$ or $\sup_{z\in\mathcal{D}}\abs{G-D(g)}\leq 2K^4\sqrt{\gamma}(N\eta)^{-\inv{2}}$. The latter is true with probability $1- \re^{-c(\log N)^2}$, since if we take the $T$ in the definition of $\mathcal{D}$ to be a large enough constant, then the former case holds with $O(\re^{-c(\log N)^2})$ probability.
\end{proof}


\begin{corollary} \label{cor:Averaged Error}
Let $a>0$ be a small constant. Then on 
\[
    \mathcal{D}' : = \{ z\in\mathcal{D}\vert \kappa\geq N^{-a}\}. 
\]
we have 
\[
    \abs{\EE{G}-D(g)}_\infty\lesssim (\log N)^{16}\left(\frac{1}{N\kappa \omega^3} + \frac{1}{(N\eta)^2 \omega^5}\right).
\]
\end{corollary}

\begin{proof}
    By integration by parts,
    \[
        -\EE{GS(G)} -\EE{G}z = I.
    \]
    Let $R= (\abs{G_{ij}-D(g)_{ij}})_{1\leq i,j\leq N}$.
\[
    -\EE{G}S(\EE{G}) - \EE{G}z = I + \EE{ (R) S(R)} = I + \EE{O(\abs{D(g)-G}_\infty^2)}.
\]
Repeating the argument in the proof of Theorem \ref{thm:bulklaw} on $\EE{G}$ instead of $G$, we have
\[
    \abs{\EE{G}-D(g)}_\infty\lesssim \EE{O(\abs{D(g)-G}_\infty^2)}\omega^{-1}.
\]
We use Theorem \ref{thm:bulklaw} and the crude bound $\Im m \leq \eta \kappa^{-2}$ to get the conclusion.
\end{proof}

\begin{remark} When we proved this local law, the only error estimates that depended strongly on the particular model we are considering are the stability results for the limiting vector equation. When considering the case of sample covariance matrices, though they are not exactly considered in the context of our proof, the stability results and the square root behavior at the right edge hold for sample covariance matrices. Thus, we will be able to prove a local law for sample covariance matrices.

\end{remark}


\subsection{Upper Bound of Top Eigenvalue}

Here we first show a lemma that combines our estimates on the average empirical spectral density with Gaussian concentration to prove upper bound for the top eigenvalue. 
\begin{lemma} \label{lem:AverageComp}
For $N\in\mathbb{N}$, consider a family of random measures $\mu_N=\frac{1}{N}\sum_{k=1}^N \delta_{\lambda_k}$ where $\lambda_1\geq \cdots\geq \lambda_N$ such that there is a deterministic $\hat{\lambda}_1$ satisfying $\lambda_1 = \hat{\lambda}_1+ o(N^{\varepsilon-\inv{2}})$ for any $\varepsilon>0$. Assume that there exists a deterministic measure $\nu$ whose Green's function satisfies
\begin{gather}\label{squareRootBehavior}
\text{Im}[m_{\nu}](x + i \eta) \le C \frac{\eta}{\sqrt{\kappa+\eta}}.
\end{gather}
where $\kappa:=\dist(\supp(\nu),x)$ and that
\begin{gather}\label{Green's function Comparison}
    |\EE{m_{\mu_N}(z)} - m_{\nu}(z)| =o(N^{-\frac{1}{2}-\gamma}).
\end{gather}
for some $\gamma>0$ and all $z= E +i\eta$ with $\dist(E,\supp(\nu)) \ge N^{-\epsilon}$ and $\eta \ge N^{-\delta-\frac{1}{2}}$ for some $\delta,\epsilon>0$.

Then, $\dist(\lambda_1,\supp(\nu))\leq N^{-\epsilon'}$ for some $\epsilon'>0$.
\end{lemma}
\begin{proof}
Assume for contradiction that $\lambda_1$ lies outside a distance $N^{-\epsilon'}$ of $\supp(\nu)$ where $\epsilon'$ is smaller than the $\epsilon$ in the condition for \eqref{Green's function Comparison}.

Notice that we have the following inequality
\begin{gather} \label{ConcenGreen}
    \frac{1}{N}
    \le C \int_I \int_{-\infty}^{\infty} \frac{1}{N} \frac{\eta  \delta_{x=\lambda_1}}{(x-E)^2 + \eta^2} \dd x \dd E \le C \int_I \text{Im}[m_{\mu_N}](E+\ri \eta) \dd E, 
\end{gather}
letting $I=[\hat{\lambda}_1-\frac{N^{\gamma'}}{\sqrt{N}}, \hat{\lambda}_1+ \frac{N^{\gamma'}}{\sqrt{N}}]$ with $\gamma'<\gamma \wedge \delta/2$.  We need to choose $\eta$ to be smaller than $\frac{1}{N^{1/2 + \gamma'}}$. This will ensure that at least a one sided $\eta$ neighborhood of $\lambda_1$ will always lie the region $[\hat{\lambda}_1 - \frac{1}{N^{1/2 +\gamma'}}, \hat{\lambda}_1 + \frac{1}{N^{1/2 +\gamma'}}]$. The integral of $\frac{\eta}{(x-z)^2 + \eta^2}$ inside this half interval of size $\eta$ will certainly be greater than $\frac{\eta^2}{2 \eta^2}= \frac{1}{2}$.

We can take the expectation of \eqref{ConcenGreen} to get,
\begin{gather} \label{GfuncCompar}
    \frac{1}{N} \le C' \int_I \Im[\EE{m_{\mu_N}}](E+\ri \eta) \dd E  
    \le C' \int_I \frac{o(1)}{N^{\frac{1}{2} +\gamma}} + C' \int_I \text{Im}[m_{\nu}](E + \ri \eta) \dd E 
    \\ \label{SquareRoot}
    \le  C' \frac{o(1)}{N} + C'' \int_I \frac{\eta}{\sqrt{\kappa}} 
    \le C' \frac{o(1)}{N} + C'' \frac{\eta N^{\gamma'}}{\sqrt{N} \sqrt{\kappa}}
\end{gather}
In $\eqref{GfuncCompar}$ we used the assumption \eqref{Green's function Comparison}(since $\hat{\lambda}_1$ lies in the region this assumption is valid) 
while in \eqref{SquareRoot}, we used the fact that $\nu$ satisfies \eqref{squareRootBehavior}.

Notice that we can set $\eta = N^{-1/2 - \delta}$ for $\delta$ positive and $\kappa = N^{-\min{(\epsilon, \delta/4)}}$ and see that the error of \eqref{SquareRoot} will be $\frac{o(1)}{N}$. This contradiction implies that for large $N$, $\hat{\lambda}_1$ must necessarily be less than $N^{-\epsilon'}$. By concentration of $\lambda_1$ around $\hat{\lambda}_1$, we would know that all $\lambda_1$ will be less than $N^{-\epsilon'}$.
\end{proof}

\begin{theorem} \label{thm:Edge upper bound for the top eigenvalue}
For the Gaussian Ensemble that we are considering, there exists an $\epsilon>0$ such that all eigenvalues lie within distance $N^{-\epsilon}$ from the edge.
\end{theorem}

\begin{proof}
We would like to apply Lemma \ref{lem:AverageComp}. First notice that by Gaussian concentration, we are able to prove that the distance of $|\lambda_1 - \EE{\lambda_1}| \le \frac{(\log N)^2}{\sqrt{N}}$ with probability $1-O(N^{-c\log N})$. We thus put $\hat{\lambda}_1 = \EE{\lambda_1}$ in the assumption of Lemma \ref{lem:AverageComp}.

Then we check that the error bounds in Corollary \ref{cor:Averaged Error} are sufficient for our purposes. The error that appears there is $\abs{\EE{G}-D(g)}_\infty\lesssim (\log N)^{16}\left(\frac{1}{N\kappa \omega^3} + \frac{1}{(N\eta)^2 \omega^5}\right)$. By the definition of $D$ and the Lipschitz continuity of $g$, we have $|\EE{\inv{N}{\Tr G}} - m_\nu |= O(N^{-\inv{2}-\gamma})$ for some $\gamma>0$  as long as we have $\eta\gg N^{-3/4+\delta}$ and $\kappa \sim N^{-\epsilon}$ for $\epsilon$ very small and $\delta>0$. Since $\delta$ can be arbitrarily small, we may choose $\eta$ such that $N^{-3/4+\delta}\ll \eta \ll N^{-1/2}$ and we can apply Lemma \ref{lem:AverageComp}.
\end{proof}

\section{Universality}

In the previous section, we proved a local law for $m_N$ as well as an improved local law for $\EE{m_N}$, and combining it with the concentration of the top eigenvalue to prove an upper bound on the top eigenvalue. According to a recent result by Landon and Yau \cite{Yau2017} below, the local law with upper bound on the top eigenvalue is sufficient to prove universality near the edge. 

\begin{theorem} \label{YauEdge}
Let $\eta* = N^{-\phi^*}$ for some $0<\phi^*<\frac{2}{3}$. We call a deterministic matrix $V$ $\eta^*$-regular if it satisfies the following properties.

\begin{enumerate}
\item There exists a constant $C_{V} \ge 0$ such that
\begin{gather*}
\frac{1}{C_V} \frac{\eta}{\sqrt{|E|+\eta}} \le \text{Im}[m_V(E + i \eta)] \le C_{V} \frac{\eta}{\sqrt{|E| + \eta}}, -1 \le E \le 0, \eta^*\le \eta \le 10,
\end{gather*}
and
\begin{gather*}
\frac{1}{C_V} \sqrt{|E|+\eta} \le \text{Im}[m_V(E+ i\eta)] \le C_{V} \sqrt{|E| + \eta},  0 \le E \le 1 , (\eta^*)^{1/2}|E| + \eta^* \le \eta \le 10.
\end{gather*}
\item There exists no eigenvalue of V in the region $[-\eta^*,0]$.
\item We have $\|V\|\le N^{C_V}$ for some $C_V >0$.
\end{enumerate}

Consider the ensemble $V_t =V + \sqrt{t} G$. Where $G$ is an independent GOE ensemble. Let $t$ satisfy $N^{-\epsilon} \ge t \ge N^{\epsilon} \eta^*$ and let $F:\mathbb{R}^{k+1} \rightarrow \mathbb{R}$  be a test function such that $\|F\|_{\infty} \le C $ and $\|F\|_{\infty} \le C$. Then there are deterministic parameters $\gamma_0\sim 1$ and $E_{-}$ such that
\begin{gather*}
|\mathbb{E}[F(\gamma_{0} N^{2/3}(\lambda_{i_0} -E_{-}), ...\gamma_{0}N^{2/3}(\lambda_{i_k} - E_{-}))] -
\mathbb{E}_{GOE}[F(N^{2/3}(\hat{\lambda}_{1} +2),...N^{2/3}(\hat{\lambda}_{k} +2))] \le N^{-c}
\end{gather*}
The first expectation is with respect to the eigenvalues of the ensemble $V_t$. The latter expectation is taken with respect to the eigenvalues $\hat{\lambda}_{i}$ of a GOE. $i_0$ is the first index i such that $i$th smallest eigenvalue of $V$ is greater than $-\frac{1}{2}$.
\end{theorem}

Call $H$ the ensemble with correlation structure $\xi_{ijkl} $.
Theorem \ref{thm:Edge upper bound for the top eigenvalue} combined with \ref{thm:bulklaw} shows that there exists a parameter $\Phi>0$ such that with high probability a matrix $M$ produced by $H$ would be $\eta^*$ regular for any $N^{-\phi}$ such that $\phi<\Phi$. Now fix some $\phi$ sufficiently small and $\phi<\Phi$ and $t=N^{-\phi}$; we would like to write the ensemble $H$ as $H'+\sqrt{t}G$ for $G$ an independent GOE ensemble. We will use the fact that $\phi$ is sufficiently small in the following section.

When N is large enough, $H'$ is the ensemble with correlation structure given by $\xi_{ijkl}- t \delta_{ij=kl}$. With t sufficiently small, the covariance matrix is positive and one can construct the ensemble. Also note that a matrix produce from $H'$ would satisfy the regularity estimates with parameter $N^{-\phi}$ as well due to our proof of the local law and upper bound for the top eigenvalue of the edge.

We will apply \eqref{YauEdge} as follows. Any matrix, $M$, in $H$ can be written in the form $M'+ t \text{GOE}$ where $M'$ is a matrix produced from the ensemble $H'$.$M'$ can be diagonalized into the form $V'$ by some unitary transformation $U$, which will leave invariant the GOE part. We will then condition on this matrix $M'$ and apply theorem \eqref{YauEdge}.We will thus get the following statement, the matrices of the form $M'+t \text{GOE}$ will satisfy a universality statement of the form.

\begin{gather} \label{BadScale}
    |\mathbb{E}_{M'}[F(\gamma_{0} N^{2/3}(\lambda_{1} -E_{-}), ...\gamma_{0}N^{2/3}(\lambda_{k} - E_{-}))] - 
\mathbb{E}_{GOE}[F(N^{2/3}(\hat{\lambda}_{1} +2),...N^{2/3}(\hat{\lambda}_{k} +2))] \le N^{-c}
\end{gather}

 where $\lambda_{1}$ are the eigenvalues of the considered matrix $M' + t \text{GOE}$.$\mathbb{E}_{M'}$ denotes the conditional expectation over the matrices of the form $M'+ t \text{GOE}$. We used for $N$ large enough, the largest eigenvalue of $M'$ is of distance less than $1/2$ from the edge, so the index $i_0$ is 1. The only issue with \eqref{BadScale} is that $\gamma_0$ is a function of the initial data, we will make this a universal constant in the next section.

 \subsection{Changing the scaling factor}

 Let $a^{t}$ be the edge of the ensemble corresponding to $H' + t \text{GOE}$ . Let us denote the Green's function of this ensemble as $m_{(H')^t}$; let us also write the density of this ensemble as $\rho_{(H')^t}$

As in Thm 2.2, let $E^{t}_{-}$ be the edge corresponding to the model $V + t \text{GOE}$ where $V$ is the deterministic diagonal matrix and the $\text{GOE}$ is an independent ensemble. Let us denote the Green's function of this ensemble as $m_{V^t}$; let us also write the density of this ensemble as $\rho_{V^t}$. We will be considering the case that $V$ is a fixed matrix coming from the ensemble $H$. From now on, we will assume that $V$ is $\eta^*$ regular so that the conditions of Thm 2.2. hold.

From the results of Thm 2.2. we can write $\rho_{V^t}(E) = \gamma_{V^t}^{-1/2} \sqrt{E- E^{t}_{-}}(1+ t^{-2} O(|E-E^{t}_{-}|))$
and $\rho_{H^t}(E) = \gamma_{H^t}^{1/2} \sqrt{E- a^{t}}(1+ t^{-2} O(|E-E^{t}_{-}|))$.

We will show the following bound on the sacling factors.
\begin{lemma} \label{ScalingFactor}
For sufficiently large N, we have that $\gamma_{H} - \gamma_{V^t} =\gamma_{(H')^{t}} - \gamma_{V^t} = O(t)$
\end{lemma}

\begin{proof}
Define $z_1$ to be the solution of $z_1 + t m_{(H')^0}(z_1) = a^t + \kappa$ and $z_2$ to be the solution of $z_2 + t m_{V^0}(z_1) = E_{-}^{t} + \kappa$. We would need to compare the values of $\text{Im} m_{(H')^0}[z_1]$ and $\text{Im} m_{V^0}[z_2]$ in order to compare the values of $\rho_{(H')^t}$ and $\rho_{V^t}$ at a distance $\kappa$ away from their edges.

Namely, we know that $\rho_{(H')^t}(a^t + \kappa) = \text{Im}[m_{(H')^{0}}[z_1]]$ and similarly for $z_2$.

Indeed, we have

\begin{gather}
\pi[\rho_{(H')^t}(a^t + \kappa) - \rho_{V^t}(E^t_{-} + \kappa)] = \text{Im}[m_{(H')^{0}}[z_1]] - \text{Im}[m_{V^{0}}[z_2]]\\ \label{Differences}
= \text{Im}[m_{(H')^{0}}[z_1]] - \text{Im}[m_{V^{0}}[z_1]]  + \text{Im}[m_{V^{0}}[z_1]] - \text{Im}[m_{V^{0}}[z_2]]
\end{gather}

In \eqref{Differences}, the first term can be bounded by a sufficiently good local law. The second term can be bounded by a Lipschitz condition provided $|z_1 - z_2|$ are sufficiently close to each other.

We will now attempt to bound the quantity $|z_1 -z_2|$

\begin{lemma}\label{Distance Lemma} 
Assume that we are considering a matrix model $H$ that is $\eta^*$ regular for all $\eta^* = N^{-\phi}$, for $\Phi>\phi>0$.

Consider the time scale $t= N^{-\phi/2}$ and choose $\kappa$ to be the almost optimal $t^{2+\epsilon}$ for the edge expansion. Then there exists a small parameter $\delta$ such that for N large enough we can ensure that $|z_1 -z_2| \le t^{2+\delta}$
\end{lemma}

\begin{proof}

We have that
\begin{gather}
z_1 + t m_{(H')^0}(z_1) - (z_2 + t m_{V^{0}}(z_2)) = (a^t - E^{t}_{-}) \\ \label{Rouche}
(z_1 -z_2) + t(m_{(H')^{0}}(z_1) - m_{(H')^{0}}(z_2)) = (a^t -E^{t}_{-}) + t(m_{V^{0}}(z_2) - m_{(H')^{0}}(z_2))
\end{gather}

We will try to prove that $|z_1 - z_2|$ is sufficiently small. We will do this by appealing to Rouche's Theorem and a Local Law bound to the second term on the RHS of \eqref{Rouche}.

We will now address the Local Law portion of the above estimate. Recall the formula that $\Im[z_1]= t \text{Im}[m_{(H')^{t}}(a^t + \kappa)]$. From the earlier expansion of the density around $\kappa$, we know that for $\kappa \le ct^2$, we have that $\text{Im}[m(z_1)]$ is up to a constant factor equal to $\gamma_{(H')^{t}}\sqrt{\kappa}$ where the $\gamma$ scaling factor is of order 1.

Thus, we see that $\Im[z_1]$ is of the order of $t\sqrt{\kappa}$. Notice that if we take $\kappa$ near the limit scale of $t^{2+\epsilon}$, as we will do later, then we will have that $\text{Im}[z_1]$ is of the size $t^{2+\epsilon/2}$ Using the fact that we are dealing with time scales of the order $t= N^{-\phi/2}$, we see that $\Im[z_1] = N^{-\phi(1+\epsilon/4)}$. This is in a regime where we can apply the local law \ref{thm:bulklaw}. 

To confirm this carefully, note  that 
$\dist(z_1,\supp \nu) \ge \text{Im}[z_1]$
so the following should hold for 
\begin{gather*}
\text{Im}[z_1] = N^{-\phi(1+\epsilon/4)}  \ge (\log N)^{10} N^{-1}(N^{2/3 (-\phi)(1+\epsilon)})^{-4} \gg (\log N)^{\log\log N} N^{-1/2}  
\end{gather*}
so the point $z_1$ is in the region $\mathcal{D}$ when  when we have that $\phi$ is sufficiently small.
Clearly, we would also have that a circle of radius $t^{2+\delta}$ around $z_1$ for $\delta > \epsilon/2$ would also lie in the region $\mathcal{D}$ .

Applying \ref{thm:bulklaw} for z in a circle of radius $t^{2+\delta}$ around $z_1$ will give us that the error of $|m_{V^{0}}(z) -m_{(H')^{0}}(z)| \le (\log N)^4(\sqrt{1 /(N\text{Im}(z))}) (\text{Im}(z))^{- 2/3}$. This can be seen to be much less than $t^3$ given that we set $\phi$ to be sufficiently small. Thus, we have a good local law bound on the second term of \eqref{Rouche} once $\phi$ is set to be sufficiently small.

We know that since we assumed $V$ is $\eta^*$ regular for $\eta^* = N^{-\Phi}$ for $\phi<\Phi$ from the local law on the ensemble $H$, we also know that with high probability $|a^t - E^t|$ should be less than $N^{-\Phi}$. Again choosing $\phi$ small enough, this will imply that $|a^t - E^t| \le t^{3}$

Consider a circle of radius equal to $R= \frac{t^{2+\delta}}{1- t K}$ where $K$ is such that we have  $|m_{(H')^{0}}(z_1) - m_{(H')^{0}}(z_2)| \le K |z_1 -z_2|$.
around the point $z_1$. Notice that t decreases as N increases; thus for very large N, we will have that $tK \le \frac{1}{2}$. Therefore, we have that R is a circle of radius less than $2 t^{2+\delta}$ for large enough N.

On this circle of radius R, we have by the local law and estimates on $|a^{t}- E^{t}_{-}|$ that the right hand side of \eqref{Rouche} will be less than the left hand side of \eqref{Rouche} in absolute value on the boundary. If the left hand side of $\eqref{Rouche}$ were 0, then we would clearly have the unique solution $z_2=z_1$. Rouche's theorem then shows that there is a solution such that $|z_2-z_1| \le R= t^{2+\delta}$\end{proof}. Putting this content back into \eqref{Differences} with $\kappa = t^{2+\epsilon}$.

\begin{gather*}\label{KappaCompare}
\gamma_{H^t}^{1/2} t^{1+\epsilon/2} (1+ t^{-2} O(t^{2+\epsilon}))-\gamma_{V^t}^{-1/2} t^{1+\epsilon} (1+ t^{-2} O(t^{2+\epsilon}) \le \\ \label{LLandLip}
\text{Im}[m_{(H')^{0}}[z_1]] - \text{Im}[m_{V^{0}}[z_1]]  + \text{Im}[m_{V^{0}}[z_1]] - \text{Im}[m_{V^{0}}[z_2]] \le
t^3 + K t^{2+\delta}
\end{gather*}

For the first term in \eqref{LLandLip}, we used the local law around $z_1$ to bound the quantity by $t^3$ for the second quantity we used Lipschitz continuity of $m_{V^{0}}$ combined with the estimate on $|z_1-z_2|$ coming from \eqref{Distance Lemma}
Notice that if we now have that $|\gamma_{V^t}^{1/2}- \gamma_{H^t}^{1/2}| \ge t $, then it would clearly be impossible for the inequality in $\eqref{KappaCompare}$ to hold.
Thus, we have proved a bound on 
$|\gamma_{V^t}^{1/2}- \gamma_{H^t}^{1/2}| \le t$, which can be turned into an o(1) bound on $\gamma_{V^t} -\gamma_{H^t}$ by squaring and using the fact that $\gamma_{V^t}$ is of constant order.
\end{proof}

\subsection{Final universality Result}

Using the scaling results coming from the previous section we can translate \eqref{BadScale} as follows.

\begin{theorem}\label{AlmostUniversality}

There exists a scaling factor $\gamma$ that depends only on the matrix ensemble H such that the following inequality holds for functions $G: \mathbb{R}^{k} \rightarrow \mathbb{R}$ such that $\|G\|_{\infty}, \|\nabla G\|_{\infty} \le C$

\begin{gather}\label{AlUeq}
 |\mathbb{E}_{H}[G(\gamma N^{2/3}(\lambda_{2} -\lambda_{1}), ...\gamma N^{2/3}(\lambda_{k} - \lambda_1))] - 
\mathbb{E}_{\text{GOE}}[G(N^{2/3}(\hat{\lambda}_{2}-\hat{\lambda}_1),...N^{2/3}(\hat{\lambda}_{k} -\hat{\lambda}_1))] \le N^{-c}
\end{gather}
\end{theorem}
\begin{proof}

First, notice that we can find a function $F:\bR^{k+1} \rightarrow \bR$ such that $\|F\|_{\infty}$ and $\|\nabla F\|_{\infty}$ are bounded  and
\begin{gather*}
    F(x_1,..x_{k+1}) = G(x_1-x_2,...x_1-x_{k+1})
\end{gather*}

Recall from earlier discussion that we can write any matrix from the ensemble $H$ as $M' + t GOE$ where M' is generated from the ensemble $H'$ with correlation structure $\xi_{abcd} - t^2 \delta_{ab=cd}$
Let $\Omega$ be the set in which we know that $M'$ has sufficiently good regularity so that $\eqref{BadScale}$ holds for the function F.
On $\Omega$, we would like to change the scaling factor $\gamma_0$ to $\gamma$, which is the scaling factor at the edge for the limiting spectral density.
As before, with high probability $M'$ has sufficient regularity so we can ensure that \eqref{BadScale} holds. We only need to change the $\gamma_0$ factor to $\gamma$ which the edge scaling coefficient of the ensemble $H$.

From \eqref{ScalingFactor}, we know that the difference between the $\gamma_0$ appearing in \eqref{BadScale} and the $\gamma$ appearing here is of the order $t = N^{-\phi/2}$.
Finally, one can appeal to the Lipschitz nature of $F$ as well as the fact that the $N^{2/3}(\lambda_{i_k} - E_{-})$
are bounded to say that

\begin{gather*}
    |F(\gamma N^{2/3}(\lambda_{1} -E_{-}), ...\gamma N^{2/3}(\lambda_{k} - E_{-}))-
    F(\gamma_0 N^{2/3}(\lambda_{3} -E_{-}), ...\gamma_0 N^{2/3}(\lambda_{k} - E_{-}))| \le  C k N^{-\phi/2} 
\end{gather*}

One can then take expectation with respect to the ensemble $M' + t \text{GOE}$ with $M'$ fixed and then apply the triangle inequality with respect \eqref{BadScale} to prove

\begin{gather*}
    |\mathbb{E}_{M'}[F(\gamma N^{2/3}(\lambda_1 - E^{M}_{-}) ,...,\gamma N^{2/3}(\lambda_k -E^{M}_{-})] - \mathbb{E}_{\text{GOE}}[F(N^{2/3}(\hat{\lambda}_1+2),...N^{2/3}(\hat{\lambda}_k +2))]| \le N^{-c}
\end{gather*}

Translating this statement to $G$, we get for matrices $M'$ in $\Omega$

\begin{gather}\label{CondUniv}
    |\mathbb{E}_{M'}[G(\gamma N^{2/3}(\lambda_1 - \lambda_2) ,...,\gamma N^{2/3}(\lambda_1 -
    \lambda_k)] - \mathbb{E}_{\text{GOE}}[G(N^{2/3}(\hat{\lambda}_1-\hat{\lambda}_2),...N^{2/3}(\hat{\lambda}_1 -\hat{\lambda}_k))]| \le N^{-c}
\end{gather}

One would now like to remove the conditional expectation in the above expression. Namely, we would like to integrate \eqref{CondUniv} over the matrices $M'$ found in $\Omega$ while using the trivial bound that $|E_{H'}[G] - E_{GOE}[G]|$ is bounded by a constant for all matrices $M'$ not found in  $\Omega$.
We thus get the full universality statement

\begin{gather}\label{FullUniv}
    |\mathbb{E}_{H}[G(\gamma N^{2/3}(\lambda_1 - \lambda_2) ,...,\gamma N^{2/3}(\lambda_1 -
    \lambda_k))] - \mathbb{E}_{\text{GOE}}[G(N^{2/3}(\hat{\lambda}_1-\hat{\lambda}_2),...N^{2/3}(\hat{\lambda}_1 -\hat{\lambda}_k))]| \le N^{-c}
\end{gather}

as desired.

\begin{remark}
As long as we know that a version of the Dyson-Brownian Motion result holds for sample covariance matrices, then we will be able prove edge universality using the local law and edge upper bound for the top eigenvalue results from the previous section.
\end{remark}
\end{proof}

\appendix
\section{Proof of Theorem \ref{Jaffard}}
Let $B=I-A$. Since $\norm{B}<1$, We can expand $A^{-1} = \sum_{k=1}^\infty B^k$. We need the following lemma to bound each term.

For simplicity, we will prove the statement of polynomial decay of inverse of order $1$ for matrix decay of order $2+\delta$. The following proof can readily be generalized to show decay of inverse of order $d-1- \delta$, $\delta>0$, given matrix decay of order $d$ for $d>2$.

\begin{lemma}
We have that 
\begin{gather} \label{pwrnrmbnd}
\|B^n\|_{\alpha} \le E n^k (\frac{1+\norm{B}}{2})^{n}
\end{gather}
where E is a function that, upon fixing $\delta$ is only polynomially dependent on $\norm{B}_{2+\delta}$ and $1-\norm{B}$ while k is dependent only on $\delta$.
\end{lemma} 
\begin{proof}
We want to compute the entries of $[B^{n}]_{jk}$. We will now define two auxiliary matrices $[\tilde{B}]_{xy} = B_{xy} \chi[|x-y|\le \frac{j-k}{n}]$ and $[\hat{B}]_{xy} = \frac{j-k}{n} B_{xy}\chi[|x-y|\ge \frac{j-k}{n}]$.

Notice that we have the following identity
\begin{gather}\label{JaffardSum}
    |j-k| [B^{n}]_{jk} = n \sum_{i=0}^{n-1} (\tilde{B})^{i} \hat{B} B^{n-i-1}
\end{gather}

We now use the following interpolation identity which appears in \cite{Jaffard}

\begin{lemma}\label{MainTermJaffard}
If $\|M\|_{l^2} \le \infty$ and $\|N\|_{l^2} \le \infty$, then we have that
\begin{gather}\label{JaffardEq}
|(M \hat{B} N)_{xy}| \le \|M\|_{l^2} \|B\|_{2+\delta} \|N\|_{l^2}
\end{gather}
\end{lemma}

\begin{proof}
Notice that the decay of $\hat{B}$ is order $1+\delta$ with coefficient $\norm{B}_{2+\delta}$. Thus we can say that $\hat{B}$ exists in $l^{q}$ for $q\ge\frac{1}{1+\delta}$.
Also see that $|(M \tilde{B} N)_{xy}| = |<M e_x, \tilde{B}N e_y>|$ where $e_x$ is the canonical basis of our matrix space.
By Young's inequality, we can say that

\begin{gather} \label{Young}
\|\hat{B}N e_y\|_{l^2} \le \|B\|_{2+\delta} \|Ne_y\|_{l^2} \le \|B\|_{2+\delta} \|N\|_{l^2}
\end{gather} 
which we can do since we have that  $r= \frac{1}{2} = q+p -1 =1 + \frac{1}{2} -1$ where we are allowed to set $q=1$.
We finally apply the Cauchy-Schwarz inequality to $|<M e_x, \tilde{B}N e_y>| \le \|M\|_{l^2} \|B\|_{2+\delta} \|N\|_{l^2}$
\end{proof}

Applying the above lemma to each term of the form $\tilde{B^{i}}\hat{B}B^{n-i-1}$, we will be able to say that
$[\tilde{B^{i}}\hat{B}B^{n-i-1}]_{ij} \le \norm{\tilde{B}}^{i} \norm{B}_{2+\delta} \norm{B}^{n-i-1}$.
Finally, we would like to relate $\norm{\tilde{B}}$ back to $\norm{B}$.
By triangle inequality, this would amount to estimating $\frac{n}{|j-i|}\norm{\hat{B}}$. Notice that in the proof of \eqref{MainTermJaffard}, we used that $\norm{\hat{B}}\le \norm{B}_{2+\delta}$.

Thus, to get that $\norm{\hat{B}}$ is sufficiently close to $\norm{B}$, we would need to assume a few conditions on $|i-j|$. Clearly, there exists a constant $C$ large enough that if we assume that $|j-i| > n \frac{2 \norm{B}_{2+\delta}}{1-\norm{B}}$, then we would know that $\norm{\tilde{B}} \le \norm{B} +\frac{1-\norm{B}}{2} = \frac{1+\norm{B}}{2}$.

Assuming this condition on $|j-i|$, we find that $[\tilde{B^{i}}\hat{B}B^{n-i-1}]_{ij} \le(\frac{1+\norm{B}}{2})^{n-1} \norm{B}_{2+\delta}$.
Thus, we find that in \eqref{JaffardSum} we have a bound of $n(\frac{1+\norm{B}}{2})^{n-1} \norm{B}_{2+\delta}$
In the case that $|j-i|$ is less than $n \frac{2 \norm{B}_{2+\delta}}{1-\norm{B}}$, we find that we have $|j_i|[B^{n}]_{ij}\le n \frac{2 \norm{B}_{2+\delta}}{1-\norm{B}}$.
A trivial bound for $|i-j|[B^{n}]_{ij}$ would be a sum of the two quantities that we have derived above.
\end{proof}

With the lemma in hand, we are able to say that 
\begin{gather}\label{finalres}
\|A\|_{1} \le   \sum_{n=1}^{\infty} \|B^n\|_{1} \le E \frac{2^{k+1}}{(1-\|B\|)^{k+1}}
\end{gather}
and we are done.

\begin{remark}

If we want to show decay of inverse of order $d>\alpha>d-\frac{1}{2}$ with coefficient of decay dependent only polynomially on $\norm{A}_{d}$ and $\norm{I-B}$, then we would need a better interpolation result as appears in \cite{Jaffard}.

The main issue is that we are no longer able to estimate quantities like $<Me_i|\tilde{B}N e_j>$ in $\eqref{MainTermJaffard}$ using the $l_2$ norms of $M$ and $N$ and instead one must use
the $l_p$ norms of $M$ and $N$ for $p$ between $1$ and $2$.

One must then interpolate the $l_p$ norm of $M$ and $N$ of with the $l_2$ norm and the appropriate $\alpha$ norm like

\begin{gather}\label{InterpolateJaffard}
\|B\|_{l^p} \le c_p \|B\|_{1}^{\frac{2}{p}-1} \|B\|^{2-\frac{2}{p}}_{l^2}
\end{gather}

The bounding of $|j-k|^{\alpha} [B^{n}]_{jk}$ then becomes a recurrence relation.

\begin{gather}\label{Recurrence}
\|B\|_{\alpha} \le C\|B\|_{\alpha}[ \|B^{n-1}\|_{\alpha}^{\frac{2}{p}-1} \|B\|^{(n-1)(2- \frac{2}{p})} + \sum_{i=1}^{n-1} (\|B^{i}\|_{\alpha} \|B^{n-i-1}\|_{\alpha})^{2 -\frac{2}{p}} \|B\|^{(n-1)(2 - \frac{2}{p})}]
\end{gather}

If one would want to prove inductively the bound that $\|B_n\|_{\alpha} \le n^{k} R^n$ , then placing this estimate inside the double product $\|B^{i}\| \|B^{n-i-1}\|$ and applying the trivial bound that $i^k (n-i-1)^k \le n^{2k}$ we would want $n^{2k(2- \frac{2}{p})} \le n^{k}$. One notices now that this is only possible if we have that $2- \frac{2}{p} \le \frac{1}{2}$ or $p\le \frac{4}{3}$ .

We could only choose $p<  \frac{4}{3}$ if we choose $\alpha < d -\frac{1}{2}$.

If one has the comfort that $\norm{I -A}$ is bounded away from 0, then one can analyze the recursion at any order $\alpha<d$ but the growth of the alpha norm in the recursion  will no longer be $\norm{I-A}$ but some parameter $r>\norm{I-A}$

\end{remark}

\printbibliography
\end{document}